\documentclass{amsart}
\usepackage[english]{babel}
\usepackage{amsfonts, amsmath, amsthm, amssymb,amscd,indentfirst}

\usepackage{esint}
\usepackage{graphicx}
\usepackage{epstopdf}
\usepackage{amsmath,amssymb,latexsym,indentfirst}
\usepackage{times}
\usepackage{palatino}

\newtheorem{theorem}{Theorem}[section]

\newtheorem{proposition}{Proposition}[section]
\newtheorem{lemma}{Lemma}[section]

\newtheorem{corollary}{Corollary}[section]

\newtheorem{example}{Example}[section]
\newtheorem{remark}{Remark}[section]

\def\R{\mathbb{R}}

\def\bp{\begin{proof}}
\def\ep{\end{proof}}

\def\R{{\cal R}}

\def\R{\mathbb{R}}

\def\B{\mathcal{B}}

\def\R{\mathbb{R}}

\def\H{\mathbb{H}}         
\def\S{\mathbb{S}}         
\def\B{\mathbb{B}}

\def\cb{\bar{\nabla}}
\def\nb{\bar{N}}
\def\x{\vec{x}}
\def\sumk{\sum_{k=1}^{n}}
\def\suml{\sum_{l=1}^{n}}

\def\gb{\bar{g}}

\def\aneleu{\mathcal{A}(r_{_1},r_{_2})}
\def\anelconf{\mathcal{A}(\bar{r}_{_1},\bar{r}_{_2})}

\begin{document}

\title[Uniqueness results for free-boundary minimal hypersurfaces]{Uniqueness results for free-boundary minimal hypersurfaces in conformally Euclidean balls and annular domains}

\author{Ezequiel Barbosa}
\address{Universidade Federal de Minas Gerais (UFMG), Departamento de Matem\'{a}tica, Caixa Postal 702, 30123-970, Belo Horizonte, MG, Brazil}
\email{ezequiel@mat.ufmg.br}
\author{Edno Pereira}
\address{Universidade Federal de Minas Gerais (UFMG), Departamento de Matem\'{a}tica, Caixa Postal 702, 30123-970, Belo Horizonte, MG, Brazil}
\email{edno@mat.ufmg.br}
\author{Rosivaldo Ant\^onio Goncalves}
\address{Universidade Estadual de Montes Claros (UNIMONTES), Departamento de Ciencias Exatas, Unimontes, Montes Claros, MG, Brazil}
\email{rosivaldo@unimontes.br}
\thanks{The authors were partially supported by  CNPq, CAPES and FAPEMIG/Brazil agency grants.}

\begin{abstract}
In this paper we prove that a flat free-boundary minimal $n$-disk, $n\geq3$, in the unit Euclidean ball $B^{n+1}$ is the unique compact free boundary minimal hypersurface in the unit Euclidean ball which the squared norm of the second fundamental form is less than either $\frac{n^2}{4}$ or $\frac{(n-2)^2}{4|x|^2}$. 
Moreover, we prove analogous results for compact free boundary minimal hypersurfaces in annular domains with a conformally Euclidean metric. 
\end{abstract}

\maketitle

\section{Introduction}\label{intro}

Free boundary minimal submanifolds are an important branch of Differential Geometry and have received much attention. A classical result due to J. C. C. Nitsche \cite{Ni} is the following: \emph{Let $D^2\subset \mathbb{R}^3$ be a proper branched minimal immersion with free boundary on the standard unit sphere, where $D^2$ is a disk. Then $D^2$ is a flat disk.} There has been much work extending this result in many different directions. For instance, Fraser-Schoen \cite{FSc} showed an extension of Nitsche's Theorem for surfaces with arbitrary codimension. However, there is no complete version of Nitsche's theorem for submanifolds with high dimension. From another side, as an  application of Nitsche's Theorem and Gauss-Bonnet Theorem we get the following: \emph{Let $\Sigma^2\subset \mathbb{R}^3$ be a compact free boundary minimal immersion in the standard unit Euclidean ball such that $|A|^2\leq4$, where $|A|^2$ denote the squared norm of the second fundamental form of $\Sigma$. Then $\Sigma^2$ is a flat disk.}  In order to see that, use the Gauss-Bonnet Theorem to obtain
\[
2\pi(2-2g-r)=\int_{\Sigma}(2-\frac{|A|^2}{2})da\geq0\,,
\]
where $g$ denotes the genus of $\Sigma$ and $r$ the number of its boundary components. Hence, either $g=0$ and $r=1$, or $g=0$ and $r=2$. If $g=0$ and $r=1$, it follows from Nitsche's Theorem that $\Sigma$ is a flat disk. If $g=0$ and $r=2$, then $|A|^2=4$. Applying Simon's equation $\Delta |A|^2=-2|A|^4+4|\nabla |A||^2$, which is valid for minimal surfaces in the 3-dimensional Euclidean space, we get a contradiction. In \cite{CMV}, M. Cavalcante, A. Mendes and F. Vit\'orio prove the same result considering high codimensional surfaces $\Sigma^2\subset B^{2+k}$. We point out that the constant $c=4$ is not the better one for that result. With a convergence argument, and using the result above with the inequality $|A|^2\leq4$, we can slightly improve that constant: there exists a positive $\varepsilon_0$ such that the only free boundary minimal surface $\Sigma\subset B^{2+k}$ satisfying $|A|^2\leq4+\varepsilon_0$ is the flat disk. In fact, if that is false, for each $\varepsilon>0$ such that $\varepsilon \rightarrow0$ there exists a free boundary minimal surface $\Sigma_{\varepsilon}$ which is not totally geodesic and 
\[
|A|^2\leq4+\varepsilon\,.
\] 
As the second fundamental form is uniformly bounded there exists a free boundary minimal surface $\Sigma$ such that, up to a subsequence, $\Sigma_{\varepsilon} \rightarrow\Sigma$ and the second fundamental form of $\Sigma$ satisfy $|A|^2\leq4$. Hence, $\Sigma$ is a flat disk. As  $\Sigma_{\varepsilon} \rightarrow\Sigma$, we obtain that $\Sigma_{\varepsilon}$ is topologically a disk, for $\varepsilon$ small enough. Then, from the Nitsche-Fraser-Schoen's uniquenessTheorem, we obtain a contradiction.

Another application of Nitsche-Fraser-Schoen's Theorem is the following: \emph{Let $\Sigma^2\subset \mathbb{R}^{2+k}$ be a compact free boundary minimal immersion in the standard unit Euclidean ball such that $\int_{\Sigma}|A|^2da\leq4\pi$, then $\Sigma^2$ is a flat disk.} In fact, assume that $\Sigma$ has genus $g$ and $r$ boundary components. It follows from the Gauss-Bonnet Theorem that
\[
2\pi(2-2g-r)=\int_{\Sigma}(2-\frac{|A|^2}{2})da\geq2|\Sigma| - 2\pi\,.
\]
The result follows from the fact that $|\Sigma|\geq \pi$ and equality holds only for the flat disk (see, for instance, S. Brendle \cite{Br}, and Fraser-Schoen \cite{FSc2}), since this implies that $g=0$ and $r=1$. Note that this also shows that there exists no compact free boundary minimal surface in $B^{2+k}$ with $\int_{\Sigma}|A|^2da=4\pi$. As a direct consequence of this we have that: if $\Sigma$ is not totally geodesic then 
\[
||A||^2_{\infty}|\Sigma|>4\pi(2g+r-1)\,. 
\]
Therefore, the set of free boundary minimal surfaces satisfying $||A||^2_{\infty}|\Sigma|\leq C$, for some positive constant $C$, is a compact set. Again, with a convergence argument, we can slightly improve this: there exists a positive $\varepsilon_0$ such that the only free boundary minimal surface $\Sigma\subset B^{2+k}$ satisfying $\int_{\Sigma}|A|^2da\leq4\pi+\varepsilon_0$ is the flat disk.  

In the unit sphere, we also have similar results: Let $\Sigma^2 \subset B^{2+k}$ be a free boundary minimal surface, where $B^{2+k}$ is a geodesic ball contained in a hemisphere of $\mathbb{S}^{2+k}$. If either $|A|^2\leq2$ or $\int_{\Sigma}|A|^2da\leq4\pi+2|B^2|$, where $|B^2|$ is the area of a geodesic disk with the same radius as $B^{2+k}$ then $\Sigma$ is a totally geodesic disk.  The proof follows the same lines of the above discussion and the results in \cite{BMc}.

In the Hyperbolic space $\mathbb{H}^{2+k}$, similar results are obtained considering
\[
|A|^2\leq\frac{1}{4}\left(\sinh(r_q)\log\left(\frac{e^{r_q}+1}{e^{r_q}-1}\right)  \right)^{-2}
\]
where we are denoting with $r_q$ the extrinsic distance from a point $q\in\mathbb{H}^{2+k}$ evaluated along $\Sigma^2$.

Our first result is on that question concerning the extension of those results to high dimension:

\begin{theorem}\label{main1}
Let $B^{n+1}$, $n\geq3$, the unit ball in the Euclidean space $\mathbb{R}^{n+1}$.  Let $\Sigma^n \hookrightarrow B^{n+1}$ be a compact free boundary minimal hypersurface immersed in $B^{n+1}$. If either the inequality 
\begin{equation*}
\left|A\right|^2 \leq \frac{n^2}{4}\quad \mbox{or}\quad |x|^2\left|A\right|^2 \leq \frac{(n-2)^2}{4}
\end{equation*}
is satisfied on $\Sigma$, then $\Sigma$ is a totally geodesic disk $D^n$.
\end{theorem}

Very recently, M. Cavalcante, A. Mendes and F. Vit\'orio \cite{CMV}, using a different method, has been able to prove topological results for high codimensional free boundary submanifolds in the unit Euclidean ball considering an explicit upper bound for the number $||\Phi||^2:=|A|^2-n|\vec{H}|^2$, where $\vec{H}$ is the mean curvature vector.

As the Nitsche's Theorem has also a valid version in some Euclidean conformal spaces, we can ask about the validity of a version of the above result for a class of Euclidean conformal spaces. Let $(\mathbb{B}_R^{n+1},\bar{g})$, $n\geq3$, be a Euclidean ball with radius $R$ and centred at origin, and with a conformal metric $\bar{g}=e^{2h}\left\langle ,\right\rangle$, where $R<\infty$ or $R=\infty$ and $h(x)=u(\left|x\right|^2)$, $u:[0,R^2) \rightarrow \R$ being a smooth function. Note that if $r$ is the Euclidean distance from a point $x \in \mathbb{B}_R^{n+1}$ to the origin then the distance $\bar{r}$, from $x$ to the origin with respect to the conformal metric $\bar{g}$, is given by
\begin{equation*}
\bar{r}=rI(r)\,, 
\end{equation*}
where $I(r)=\displaystyle \int_{0}^{1} e^{u(t^2r^2)} dt$. We write $B_{\bar{R}}$ to denote $(\mathbb{B}_R^{n+1},\bar{g})$, and $\bar{r}$ to denote the distance function. Suppose that function $u$ satisfies the following conditions:\\
$i) u''(\left|x\right|^2)-u'(\left|x\right|^2)^2\leq0 $ \\
$ii)  -u''(\left|x\right|^2)\left|\vec{x}\right|^2 - u'(\left|x\right|^2) \leq  0 $ \\
We can check that $B_{\bar{R}}$ is a Hadamard space (see section 2).

\begin{example} Consider $u_1:[0,\infty) \rightarrow \R$ given by $u_1\equiv0$ and $u_1:[0,1) \rightarrow \R$ given by $u_2(t)=\ln(\frac{2}{1-t})$. The functions $u_1$ and $u_2$ satisfy the conditions $i)$ and $ii)$ above. In fact, $M_1:=(\B_{\infty}^{n+1},\bar{g})$ for $\gb=e^{2u_1(\left|x\right|^2)}\left\langle \right\rangle$ is the Euclidian space $\R^{n+1}$ and $M_2:=(\B_{1}^{n+1},\bar{g})$, for $\gb=e^{2u_2(\left|x\right|^2)}\left\langle \right\rangle$, is the Hyperbolic space $\H^{n+1}$. Both are Hadamard spaces.
\end{example}

\begin{example} Consider $u:[0,\infty) \rightarrow \R$ given by $u(t)=\frac{t}{4n}$. Hence, the conditions $i)$ and $ii)$ are satisfied for all  $t$. Therefore, $(\B_r^{n+1},e^{\frac{\left|x\right|^2}{2n}}\left\langle , \right\rangle)$ with $r=\infty$ is a Hadamard space.
\end{example}

For those spaces, we obtain the following gap results.

\begin{theorem}\label{main2} Let $\Sigma^n \hookrightarrow B_{\bar{r}_0} \subset B_{\bar{R}}$ be a compact free boundary minimal hypersurface such that
\begin{equation*}
\left|A\right|^2 \leq \frac{n^2}{4\bar{r}_0^2}\,.
\end{equation*}
Then $\Sigma$ is a totally geodesic disk $D^n$ passing through the center of the ball.
\end{theorem}

\begin{theorem}\label{main3} Let $\Sigma^n \hookrightarrow B_{\bar{r}_0} \subset B_{\bar{R}}$ be a compact free boundary minimal hypersurface such that
\begin{equation*}
\left|A\right|^2 \leq \frac{(n-2)^2}{4\bar{r}^{2}} \quad \mbox{in} \quad \Sigma \setminus \vec{0}\,.
\end{equation*}
Then $\Sigma$ is a totally geodesic disk $D^n$ passing through the center of the ball.
\end{theorem}

We consider now free boundary minimal hypersurfaces immersed in a Euclidean annular domain.
For $r_{_0} \leq \infty$, consider $B_{\bar{r}_{_0}}=(\B_{r_{_0}}^{n+1},\bar{g})$, where $\bar{g}=e^{2h}\left\langle ,\right\rangle$, with $h(x)=u(\left|x\right|^2)$ and $u$ satisfying the conditions $i)$ and $ii)$. For $r_1<r_2<r_{_0}$, define the $(n+1)-$dimensional annulus $\mathcal{A}(\bar{r}_1,\bar{r}_2):= B_{\bar{r}_2} \setminus B_{\bar{r}_1}$. When we want to emphasise that the metric $\bar{g}$ coincides with the canonical metric (that is, for $u\equiv0$) we write $\mathcal{A}(r_1,r_2)$ instead of $\anelconf$.

\begin{theorem} Let $\Sigma^n \hookrightarrow \mathcal{A}(r_1,r_2)$ be a free boundary minimal hypersurface in immersed a Euclidian $(n+1)$-dimensional annulus. Assume that

\begin{equation*}
\left|A\right|^2 \leq \frac{n^2}{4r_{_2}^{2}}\,.
\end{equation*}
Then, 

$1)$ If $n = 2$, $\Sigma^n \hookrightarrow \mathcal{A}(r_1,r_2)$ is a totally geodesic annulus.

$2)$ If $n\geq3$ and $r_{_1}^2 < \frac{4(n-1)}{n^2}\,r_{_2}^2$, $\Sigma^n \hookrightarrow \mathcal{A}(r_1,r_2)$ is a totally geodesic annulus. 

$3)$ If $n\geq3$ and $r_{_1}^2 = \frac{4(n-1)}{n^2}\,r_{_2}^2$, either $\Sigma^n \hookrightarrow \mathcal{A}(r_1,r_2)$ is a totally geodesic annulus or $\Sigma^n$ is a trunk of a cone whose support cone is taking on the Clifford Torus in  $\S^n$.
\end{theorem}

Now we will study the case of a free boundary minimal hypersurface immersed in $\anelconf$, where the annulus $\mathcal{A}$ is not necessarily Euclidean. When $\Sigma^n$ is equipped with the Euclidian geometry of $\B^{n+1}_{r_{_2}} \setminus \B^{n+1}_{r_{_1}}$ we write $\Sigma^{n}_{\delta}$, and the second fundamental form will be denoted by $A_{\delta}$. If $p \in \Sigma^n$ has Euclidian distance to origin given by $r=\left|p\right|$, then the distance of $p$ to the origin with respect to the metric $\bar{g}=e^{2u(\left|x\right|^2)}\left\langle,  \right\rangle$ is given by the relation
\begin{equation*}
\bar{r}=rI(r) 
\end{equation*}
where $I(r)=\displaystyle \int_{0}^{1} e^{u(t^2r^2)} dt$. For our purposes, we define 
\[
m_0=\sup\{e^{2u(\left|x\right|^2)}; x \in \aneleu \}\,.
\]
\begin{theorem} Let $\Sigma^n \hookrightarrow \anelconf$ be a free boundary minimal hypersurface in a $(n+1)$-dimensional annulus conformal to the Euclidian annulus $\aneleu$. Assume that
\begin{equation*}
\left|A\right|^2 \leq \frac{n^2}{4\bar{r}_{_2}^{2}}\,.
\end{equation*}
Then, 

$\textbf{1})$ If $\frac{n^2}{(n-1)} \leq 4\left(\dfrac{I(r_{_2})^2}{m_0} \right)$, $\Sigma^n \hookrightarrow \anelconf$ is a totally geodesic annulus.\\ 

$\textbf{2})$ If $ \frac{n^2}{(n-1)} > 4\left(\dfrac{I(r_{_2})^2}{m_0} \right)$ and $r_{_1}^2 < \dfrac{4(n-1)}{n^2}\left( \dfrac{I(r_{_2})^2}{m_0}\right)\,r_{_2}^2$, $\Sigma^n \hookrightarrow \anelconf$ is a totally geodesic annulus.\\

$\textbf{3})$ If $\frac{n^2}{(n-1)} > 4\left(\dfrac{I(r_{_2})^2}{m_0} \right)$ and $r_{_1}^2 = \dfrac{4(n-1)}{n^2}\left( \dfrac{I(r_{_2})^2}{m_0}\right)\,r_{_2}^2$, either $\Sigma^n \hookrightarrow \anelconf$ is a totally geodesic annulus or $\Sigma^n$ is a trunk of a cone whose support cone is taking on a Clifford torus in $\S^n$ when it is being considered as $\Sigma^n \hookrightarrow \aneleu$
\end{theorem}

\section{Preliminaries and Proofs of the Main Results}

The next theorem was the motivation for establish the results of this paper. We will state it in the context that interest us, i.e,  for minimal hypersurfaces with boundary. The general case can be found in \cite{MHF}. Recall that a manifold $\bar{M}$ is said be a Hadamard space if it is complete, simply-connected and has non positive sectional curvature. 
\begin{theorem}[Mirandola-Batista-Vit\'orio, \cite{MHF}] Let $\Sigma^n \hookrightarrow \bar{M}^{n+1}$ be a compact minimal hipersurface in a Hadamard Space $\bar{M}$. Let $\bar{r}=d_{\bar{M}}(\cdot,\xi)$ be the distance in $\bar{M}$ from a fixed point $\xi \in \bar{M}$. Consider $1 \leq p < \infty$ and $-\infty<\gamma<k$. Then, for all function $0 \leq \psi \in C^1(M)$ it holds, 
\begin{eqnarray}\label{MBV}
\frac{(n-\gamma)^2}{p^p}\int_{\Sigma} \frac{\psi^p}{\bar{r}^{\gamma}}+ \frac{\gamma(n-\gamma)^{p-1}}{p^{p-1}}\int_{\Sigma} \frac{\psi^p}{\bar{r}^{\gamma}}\left|\bar{\nabla}\bar{r}^{\perp}\right|^2 &\leq& \int_{\Sigma} \frac{\left|\nabla^{\Sigma} \psi \right|^p}{\bar{r}^{\gamma-p}} \nonumber \\
&+&  \frac{(n-\gamma)^{p-1}}{p^{p-1}}\int_{\partial \Sigma} \frac{\psi^p}{\bar{r}^{\gamma-1}}\left\langle \bar{\nabla} \bar{r},\nu\right\rangle
\end{eqnarray}
futhermore if $p>1$ then equality occur if and only if $\psi\equiv 0$ on $\Sigma$
\end{theorem}
Fo the particular case where $0 \leq \gamma < n$ and $p=2$, the inequality $(\ref{MBV})$ becomes
\begin{equation}\label{MBV2}
\frac{(n-\gamma)^2}{4}\int_{\Sigma} \frac{\psi^2}{\bar{r}^{\gamma}} \leq \int_{\Sigma} \left|\nabla^{\Sigma} \psi \right|^2 \bar{r}^{(2-\gamma)} + \frac{(n-\gamma)}{2}\int_{\partial \Sigma} \frac{\psi^2}{\bar{r}^{\gamma-1}}\left\langle \bar{\nabla} \bar{r},\nu\right\rangle
\end{equation}
Check also that inequality for the choice $\gamma=0$.


We will build a model of Hadamard space where we can apply the above theorem to obtain classification results for minimal submanifolds under appropriate conditions.

Let $(\mathbb{B}_R^{n+1},\bar{g})$, $n\geq3$, be a Euclidean ball with radius $R$ and centred at origin and with a conformal metric $\bar{g}=e^{2h}\left\langle ,\right\rangle$, where $R<\infty$ or $R=\infty$ and $h(x)=u(\left|x\right|^2)$, $u:[0,R^2) \rightarrow \R$ being a smooth function. Note that if $r$ is the Euclidean distance from a point $x \in \mathbb{B}_R^{n+1}$ to origin than, the distance $\bar{r}$ from $x$ to the origin with respect to the conformal metric $\bar{g}$ is given by,
\begin{equation}\label{raioconfint}
\bar{r}=rI(r) 
\end{equation}
where $I(r)=\displaystyle \int_{0}^{1} e^{u(t^2r^2)} dt$. We write $B_{\bar{R}}$ to denote $(\mathbb{B}_R^{n+1},\bar{g})$.

\subsection{Basic relationship betwen geometries of the $(\mathbb{B}_R^{n+1},\bar{g})$ and $(\mathbb{B}_R^{n+1},\left\langle , \right\rangle)$}
We will always consider canonical cordinates on $\mathbb{B}_R^{n+1}$ and will denote by $\vec{x}$ the vector field which associate to each  point $x=(x_1,...,x_{n+1})$ the vector $\vec{x}=\sum x_i \partial_i$. Under conformal change on a riemannian metric, we have the following formulae, 

\begin{proposition}[Conformal change formulas]\label{conexaomudconf} Let $\nabla$ and $\bar{\nabla}$ be a Riemannian connection of $(M^{n+1},g)$ and $(M^{n+1},\bar{g})$ respectively, where $\bar{g}=e^{2h}g$ for some function $h:M \rightarrow \R$, $\in$ C$^{\infty}(M)$. Then, for smooth vector fields $X$ and $Y$ $\in \mathcal{X}(M)$ we have, 
\begin{flushleft}
i)$\bar{\nabla}_YX=\nabla_YX+ Y(h)X+X(h)Y - g(X,Y)\nabla h$  \\

ii)$\bar{R}(Y,Z)X=R(Y,Z)X + g(Y,X)\nabla_Z\nabla h - g(Z,X)\nabla_Y\nabla h$ \\
			    	 $\quad \quad \quad \quad \quad \quad \quad - \left\{(Hess\,h)(Z,X)-X(h)Z(h) + g(\nabla h,\nabla h)g(X,Z)  \right\}Y$\\
						 $\quad \quad \quad \quad \quad \quad \quad + \left\{ (Hess\,h)(Y,X) - Y(h)X(h) + g(\nabla h,\nabla h)g(Y,X)     \right\}Z$\\
						 $\quad \quad \quad \quad \quad \quad \quad + \left\{Y(h)g(Z,X) - Z(h)g(Y,X)\right\} \nabla h $\\

iii) $\bar{R}ic(Z,X)= Ric(Z,X) - (n-1)(Hess\,h)(Z,X) $ \\
						$\quad \quad \quad \quad \quad \quad \quad + (n-1)Z(h)X(h) - \Big\{\Delta h + (n-1)g(\nabla h,\nabla h)  \Big\} g(Z,X)$ \\

iv) $\bar{R}_{\bar{g}}=e^{-2h}\Big\{ R_g -2n\Delta h - n(n-1)g(\nabla h,\nabla h)  \Big\}$ 
\end{flushleft}
where, $\nabla h$, $(Hess\,h)$ and $\Delta h$ are calculated with respect to the metric $g$. 
\end{proposition}

\begin{lemma} Consider $\mathbb{B}_R^{n+1}$ with the Euclidean metric and denote by $\bar{\nabla}$ its connection. Then the Hessian of the function $h(x)=u(\left|x\right|^2)$ is given by
\[ 
(Hess\,h)(x)(Y,Z) = 4 u''(\left|x\right|^2)\left\langle \vec{x},Y\right\rangle \left\langle \vec{x},Z\right\rangle + 2u'(\left|x\right|^2)\left\langle Y,Z\right\rangle \,.
\]
\end{lemma}

\begin{proof} Note that the gradients vectors of the functions $h$ and $u'(\left|x\right|^2)$ are given, respectively, by
\[
\texttt{grad}\,(h) = 2u'(\left|x\right|^2)\vec{x}\quad \mbox{and} \quad  \texttt{grad}\,(u'(\left|x\right|^2)) = 2u''(\left|x\right|^2)\vec{x}\,.
\]
Hence, 
\begin{eqnarray*} 
(Hess\,h)(x)(Y,Z) &=& \left\langle \bar{\nabla}_Y \texttt{grad}\,(h),Z \right\rangle = 2 \left\langle \bar{\nabla}_Y u'(\left|x\right|^2)\vec{x},Z \right\rangle \\
									&=& 2 \left\langle Y(u'(\left|x\right|^2))\vec{x} + u'(\left|x\right|^2) \bar{\nabla}_Y \vec{x},Z \right\rangle \\
									&=& 4 u''(\left|x\right|^2)\left\langle \vec{x},Y\right\rangle \left\langle \vec{x},Z\right\rangle + 2u'(\left|x\right|^2)\left\langle Y,Z\right\rangle
\end{eqnarray*}
\end{proof}

\begin{lemma}\label{limitrici} Consider $B_{\bar{R}}=(\mathbb{B}_R^{n+1},\bar{g})$. Suppose that function $u$ satisfies $u''(\left|x\right|^2)-u'(\left|x\right|^2)^2 \leq 0$. Then, if $\nb$ is an unit vector in the tangent space $T_xB_{\bar{R}}$, we have 
\[
\bar{R}ic(\bar{N},\bar{N})(x) \leq - 4ne^{-2h}[u''(\left|x\right|^2)\left|x\right|^2+u'(\left|x\right|^2)]\,.
\]
\end{lemma}

\begin{proof} Note that if $\nb \in T_xB_{\bar{R}}$ satisfy $\gb(\nb,\nb)=1$ then the vector $N=e^{h}\nb$ is such that $\left\langle N,N\right\rangle=1$. It follows from the conformal change formulas that\begin{eqnarray*}
\bar{R}ic(\bar{N},\bar{N})&=&e^{-2h}\bar{R}ic(N,N)\\
												&=&e^{-2h}\Big\{ Ric_{_{\R^{n+1}}}(N,N) - (n-1)(Hess\,h)(N,N)+(n-1)N(h)^2  \\
												& & -\{\Delta h + (n-1)\left|\nabla h\right|^2 \}\left\langle N,N\right\rangle \Big\}\\
												&=& e^{-2h}\Big\{ -4(n-1)u''(\left|x\right|^2)\left\langle \vec{x},N\right\rangle^2 - 2(n-1)u'(\left|x\right|^2) \\
												& & + 4(n-1)u'(\left|x\right|^2)^2\left\langle \vec{x},N\right\rangle^2 \\
												& & -4u''(\left|x\right|^2)\left|\vec{x}\right|^2-2(n+1)u'(\left|x\right|^2) -4(n-1)u'(\left|x\right|^2)^2\left|\vec{x}\right|^2 \Big\}\\			 &=& e^{-2h}\Big\{ 4(n-1)[u'(\left|x\right|^2)^2-u''(\left|x\right|^2)]\left\langle \vec{x},N\right\rangle^2 - 4nu'(\left|x\right|^2) \\
												& & -4u''(\left|x\right|^2)\left|\vec{x}\right|^2-4(n-1)u'(\left|x\right|^2)^2\left|\vec{x}\right|^2 \Big\}\\
												&\overbrace{\leq}^{(*)}& e^{-2h}\Big\{ 4(n-1)[u'(\left|x\right|^2)^2-u''(\left|x\right|^2)]\left|x\right|^2 - 4nu'(\left|x\right|^2) \\
												& & -4u''(\left|x\right|^2)\left|\vec{x}\right|^2-4(n-1)u'(\left|x\right|^2)^2\left|\vec{x}\right|^2 \Big\}\\
												&=& -4ne^{-2h}\Big\{ u''(\left|x\right|^2)\left|\vec{x}\right|^2 + u'(\left|x\right|^2) \Big\}
\end{eqnarray*}
where we have used in $(*)$, the condition satisfied by $u$. Therefore,
\[
\bar{R}ic(\bar{N},\bar{N}) \leq -4ne^{-2h}[u''(\left|x\right|^2)\left|x\right|^2+u'(\left|x\right|^2)]\,.
\]
\end{proof}

\begin{lemma}\label{condhad} Consider $B_{\bar{R}}=(\mathbb{B}_R^{n+1},\bar{g})$. Suppose that function $u$ satisfies the following conditions:\\
$i) u''(\left|x\right|^2)-u'(\left|x\right|^2)^2\leq0 $ \\
$ii)  -u''(\left|x\right|^2)\left|\vec{x}\right|^2 - u'(\left|x\right|^2) \leq  0 $ \\
Then $B_{\bar{R}}$ is a Hadamard space.
\end{lemma}

\begin{proof} We just need to prove that seccional curvature $\bar{K}$ satisfies $\bar{K}(x)(\pi)\leq 0$  for any point $x \in B_{\bar{R}}$ and any plane $\pi \subset T_xB_{\bar{R}}$. Consider $\bar{E}_i$, $i=1,2$, a base for $\pi \subset T_xB_{\bar{R}}$ with $\bar{g}(\bar{E}_i,\bar{E}_j)=\delta_{ij}$. Thus, for $E_i=e^{h}\bar{E}_i$ we obtain $\left\langle E_i,E_j \right\rangle=\delta_{ij}$. It follows from the Lemma \ref{limitrici} and from the relations in the conformal change for the curvature obtained in $\textbf{Proposition \ref{conexaomudconf}}$, that
\begin{eqnarray*}
\bar{K}(x)(\pi)&=&\bar{g}(\bar{R}(\bar{E}_1,\bar{E}_2)\bar{E}_2,\bar{E}_1)\\
               &=& e^{-2h} \left\langle \bar{R}(E_1,E_2)E_2,E_1 \right\rangle \\
							 &=& e^{-2h} \Big \{ \left\langle R(E_1,E_2)E_2,E_1 \right\rangle + \left\langle E_1,E_2 \right\rangle \left\langle \nabla_{E_2}\nabla h,E_1 \right\rangle \\
							 & &   - \left\langle E_2,E_2 \right\rangle \left\langle \nabla_{E_1}\nabla h,E_1 \right\rangle  \\
							 & & -\big\{ (Hess\,h)(E_2,E_2) - E_2(h)^2 + \left|\nabla h\right|^2 \left\langle E_2,E_2\right\rangle \big \} \left\langle E_1,E_1 \right\rangle\\
							 & & +\big \{ (Hess\,h)(E_1,E_2) - E_1(h)E_2(h) + \left|\nabla h\right|^2 \left\langle E_1,E_2\right\rangle \big \}\left\langle E_2,E_1 \right\rangle\\
							 & & +\big \{ E_1(h)\left\langle E_2,E_2\right\rangle - E_2(h)\left\langle E_1,E_2\right\rangle \big \}\left\langle \nabla h,E_1\right\rangle \Big \}  \\ 
							 &=& e^{-2h}\Big\{ -(Hess\,h)(E_1,E_1)-(Hess\,h)(E_2,E_2)  \\
							 & & + E_2(h)^2 - \left|\nabla h\right|^2 +E_1(h)^2 \Big\}\\
							 &=& e^{-2h}\big\{ - 4 u''(\left|x\right|^2)\left\langle \vec{x},E_1\right\rangle^2 - 2u'(\left|x\right|^2) - 4 u''(\left|x\right|^2)\left\langle \vec{x},E_2\right\rangle^2   \\ 
							 &-&  2u'(\left|x\right|^2) + 4u'(\left|x\right|^2)^2\left\langle \vec{x},E_1\right\rangle^2 + 4u'(\left|x\right|^2)^2\left\langle\vec{x},E_2\right\rangle^2 - 4u'(\left|x\right|^2)^2\left|\vec{x}\right|^2 \big\} \\
							 &=& 4e^{-2h} \Big\{ \left[ u'(\left|x\right|^2)^2 - u''(\left|x\right|^2) \right] \big(\left\langle \vec{x},E_1\right\rangle^2 +\left\langle \vec{x},E_2\right\rangle^2 \big)\\ 
							 & &  - u'(\left|x\right|^2) - u'(\left|x\right|^2)^2\left|\vec{x}\right|^2  \Big\}
\end{eqnarray*} 

Note that,
\[
\left\langle \vec{x},E_1\right\rangle^2 +\left\langle \vec{x},E_2\right\rangle^2 = \left|x\right|^2(\cos^2(\theta_1)+\cos^2(\theta_2))\,,
\]
where $\theta_i$ denote the angle between $\vec{x}$ and $E_i$, $i=1,2$. Since the angle between  $E_1$ and $E_2$ is $\frac{\pi}{2}$, for every fixed $0\leq \theta_1 \leq \frac{\pi}{2}$, we obtain that $ \frac{\pi}{2}-\theta_1 \leq \theta_2 \leq \frac{\pi}{2}+\theta_2$, and then $\cos^2(\theta_2)\leq \cos^2(\frac{\pi}{2}+\theta_1)$. Consequently, 
\[
\cos^2(\theta_1)+\cos^2(\theta_2) \leq \cos^2(\theta_1) + \cos^2(\frac{\pi}{2}+\theta_1) \leq \cos^2(\theta_1) + \sin^2(\theta_1)=1
\]
and $\left\langle \vec{x},E_1\right\rangle^2 +\left\langle \vec{x},E_2\right\rangle^2 \leq \left|\vec{x}\right|^2 $. From the condition $i)$ we obtain
\[
\left[ u'(\left|x\right|^2)^2 - u''(\left|x\right|^2) \right] \big(\left\langle \vec{x},E_1\right\rangle^2 +\left\langle \vec{x},E_2\right\rangle^2 \big) \leq \left[ u'(\left|x\right|^2)^2 - u''(\left|x\right|^2) \right]\left|\vec{x}\right|^2 
\]
and from the condition $ii)$ follows that,  
\begin{eqnarray*}
\bar{K}(x)(\pi)&\leq& 4e^{-2h} \{ \left[ u'(\left|x\right|^2)^2 - u''(\left|x\right|^2) \right]\left|\vec{x}\right|^2 - u'(\left|x\right|^2) - u'(\left|x\right|^2)^2\left|\vec{x}\right|^2  \}\\
							 &\leq& 4e^{-2h} \{ -u''(\left|x\right|^2)\left|\vec{x}\right|^2 - u'(\left|x\right|^2) \}\\
							 &\leq& 0
\end{eqnarray*}
\end{proof}
\begin{example} Consider $u_1:[0,\infty) \rightarrow \R$ given by $u_1\equiv0$ and $u_2:[0,1) \rightarrow \R$ given by $u_2(t)=\ln(\frac{2}{1-t})$. The functions $u_1$ and $u_2$ satisfy the conditions $i)$ and $ii)$ in previous Lemma. In fact, $M_1:=(\B_{\infty}^{n+1},\bar{g})$ for $\gb=e^{2u_1(\left|x\right|^2)}\left\langle , \right\rangle$ is the Euclidian space $\R^{n+1}$ and $M_2:=(\B_{1}^{n+1},\bar{g})$, for $\gb=e^{2u_2(\left|x\right|^2)}\left\langle , \right\rangle$, is the Hyperbolic Space $\H^{n+1}$. Both are Hadamard spaces.
\end{example}

\begin{example} Consider $u:[0,\infty) \rightarrow \R$ given by $u(t)=\frac{t}{4n}$. The conditions $i)$ and $ii)$ are satisfied for all  $t$. Therefore, $(\B_R^{n+1},e^{\frac{\left|x\right|^2}{2n}}\left\langle , \right\rangle)$ with $R=\infty$ is a Hadamard space.
\end{example}

\section{GAP Results for $\left|A\right|^2$ of free boundary minimal hypersurfaces in Euclidean conformal ball} 

From now on, we consider  $B_{\bar{R}}=(\B_R^{n+1},\bar{g})$ where $\bar{g}=e^{2u(\left|x\right|^2)}\left\langle , \right\rangle$ for a function $u$ which satisfies the conditions $i)$ and $ii)$ so that $B_{\bar{R}}$ becomes a Hadamard space. For every $r<R$ we have $\B_r^{n+1} \subset \B_R^{n+1}$ and we consider $B_{\bar{r}}$ as a submanifold of $B_{\bar{R}}$. 

\begin{lemma}\label{autovposi} Let $B_{\bar{R}}$ a Hadamard space and, for $\bar{r}_0<\bar{R}$, consider $\Sigma^n \hookrightarrow B_{\bar{r}_{_0}}\subset B_{\bar{R}}$ an immersed compact minimal hypersurface with boundary $\partial \Sigma$. Consider on $\Sigma$ the operator $L:= \Delta + \left|A\right|^2 + q $ where $q:\Sigma \rightarrow $ is a nonpositive smooth function. Assume that the inequality 
\begin{equation}\label{gaphada}
\left|A\right|^2 \leq \frac{n^2}{4\bar{r}_{0}^{2}}
\end{equation}
is satisfied on $\Sigma$. Then, the first eigenvalue $\lambda_1$ of the operator $L$ for the problem
\begin{center}
$(*)\left\{\begin{matrix}
L[v]& = & -\lambda v \quad \Sigma\\ 
v& = & 0 \quad \partial \Sigma 
\end{matrix}\right.$
\end{center}
is strictly positive and consequently, if $w$ satisfy $L[w]=0$ for the problem $(*)$ above, then $w\equiv0$. 
\end{lemma} 

\begin{proof} Suppose by contradiction that $\lambda_1 \leq 0$. Let $v_1$ such that

\begin{center}
$\left\{\begin{matrix}
\Delta v_1 + (\left|A\right|^2+q)v_1 & = & -\lambda_1 v_1 \quad \Sigma\\ 
v_1& = & 0 \quad \quad \quad \partial \Sigma 
\end{matrix}\right.$
\end{center}
Then $v_1$ can be assumed positive, since the first eigenvalue of that problem is simple. It follows from integration by parts, the inequality $(\ref{gaphada})$ and $\lambda_1+q\leq0$, that  
\begin{equation}\label{estimagrad}
\int_{\Sigma} \left|\nabla v_1\right|^2 = \int_{\Sigma} -v_1\Delta v_1 d\Sigma = \int_{\Sigma}  (\lambda_1 + \left|A\right|^2+q)v_1^2 \leq \int_{\Sigma} \left|A\right|^2v_1^2 \leq \frac{n^2}{4\bar{r}_0^{2}}\int_{\Sigma} v_1^2 \\
\end{equation} 
Considering $\bar{r}$ the distance from a point $p\in \Sigma$ to the point $\vec{0}\in B_{\bar{r}}$, we have  $\bar{r}\leq \bar{r}_0$. Since $v_1\equiv0$ on $\partial \Sigma$, the inequality $(\ref{MBV2})$ for $\gamma = 0 $ gives us 
\begin{equation}\label{estima01}
\frac{n^2}{4} \int_\Sigma { v_1^2 } \leq \int_{\Sigma} \left|\nabla v_1\right|^2 \bar{r}^{2} \leq \bar{r}_0^{2}\int_{\Sigma} \left|\nabla v_1\right|^2 																										
\end{equation}
Hence, It follows from the estimate $(\ref{estimagrad})$ that
\begin{equation*}
\frac{n^2}{4} \int_\Sigma { v_1^2 } \leq \bar{r}_0^{2}\frac{n^2}{4\bar{r}_0^{2}}\int_{\Sigma} v_1^2=\frac{n^2}{4} \int_\Sigma { v_1^2 }
\end{equation*}
This means that equality occurs in $(\ref{estima01})$ and, from the result $\textbf{Theorem \ref{MBV}}$, the function $v_1$ should be the null function. This is a contradiction. Therefore, $\lambda_1>0$.
\end{proof}

\begin{remark}\label{obsteo} If we consider the condition 
\begin{equation}\label{gaphada2}
\left|A\right|^2 \leq \frac{(n-2)^2}{4\bar{r}^{2}} \quad \mbox{in} \quad \Sigma \setminus \vec{0}
\end{equation}
instead of  $(\ref{gaphada})$, where $\bar{r}=\bar{r}(x)$ denotes the distance in $B_{\bar{R}}$ \ from a point $x\in\Sigma$ to the point $\vec{0}$, and using $\gamma=2$, we obtain the same result: $\lambda_1$ is strictly positive and consequently the solution $L[v]=0$ for the problem $(*)$ above is the null function. 
\end{remark}

Note that the inequality $(\ref{gaphada})$ imposes a more rigid constraint on the total length of $\left|A\right|^2$, whereas the inequality $(\ref{gaphada2})$ just imposes a control in the growth of $\left|A\right|^2$, making the last case a less restrictive condition, allowing, for example, a situation where $\left|A\right|^2$ becomes arbitrarily large near the point $\vec{0}$.


\begin{lemma}\label{solsuport} Let $\Sigma^n \hookrightarrow B_{\bar{r}_0} \subset B_{\bar{R}}$ be a free boundary minimal hypersurface. Consider the function $v=\bar{g}(\vec{x},\bar{N})$ defined on $\Sigma$ where $\bar{N}$ denotes the normal vector field to $\Sigma$. Then, the function $v$ is solution to following problem
\begin{center}
$\left\{\begin{matrix}
L[v]&=& 0 \quad \Sigma\\ 
v & = & 0 \quad \partial \Sigma 
\end{matrix}\right.$
\end{center}
where $L[v]=\Delta v + \left|A\right|^2v+qv$ and,
$$
q:= \bar{R}ic(\bar{N},\bar{N}) + 4ne^{-2h}[u''(\left|x\right|^2)\left|x\right|^2+u'(\left|x\right|^2)]
$$
\end{lemma}
\begin{proof} Since the vector field $\vec{x}$ on $B_{\bar{R}}$ is coformal, i.e, the Lie derivative satisfies,
$$
\mathcal{L}_{\vec{x}}\bar{g}=2\sigma \bar{g}
$$ 
where $\sigma(x)=1+2u'(\left|x\right|^2)\left|x\right|^2$, the Proposition $2.1$ of \cite{CaLi} ensures that
\begin{equation}\label{lap4}
\Delta v + \left|A\right|^2v + \bar{\mbox{R}}\mbox{ic}\big(\nb,\nb \big)v = -n\nb(\sigma)
\end{equation}
Now, we just need to calculate the right side of the above equation. The gradient of the function $\sigma$ with respect to the Euclidean metric is given by,
$$
\texttt{grad}\,(\sigma) = 4[u''(\left|x\right|^2)\left|x\right|^2+u'(\left|x\right|^2)]\x
$$
but, the gradient of $\sigma$ with respect to the metric $\gb$ is given by $\cb \sigma= e^{-2h}\texttt{grad}\,(\sigma)$ thus, 
$$
\cb \sigma = 4e^{-2h}[u''(\left|x\right|^2)\left|x\right|^2+u'(\left|x\right|^2)]\x
$$
therefore,
\begin{equation}\label{derivnorm}
\nb (\sigma) = \gb\big (\cb \sigma, \nb \big) = 4e^{-2h}[u''(\left|x\right|^2)\left|x\right|^2+u'(\left|x\right|^2)]\gb(\x,\nb)
\end{equation}
replacing $(\ref{derivnorm})$ on $(\ref{lap4})$ we have
$$
\Delta v + \left|A\right|^2v + \bar{\mbox{R}}\mbox{ic}\big(\nb,\nb \big)v  +4ne^{-2h}[u''(\left|x\right|^2)\left|x\right|^2+u'(\left|x\right|^2)]v  = 0
$$ 
as desired. Furthermore, $v\equiv 0$ on $\partial \Sigma$. Indeed, since $\Sigma$ is free boundary, we get $\tilde{N} = \nu$ where $\tilde{N}$ denotes the normal vector to $\partial B_{r_0}$ and $\nu$ is a co-normal to $\Sigma$ on $\partial \Sigma$. Because $\vec{x}//\tilde{N}=\nu$ on $\partial B_{r_0}$ and $\bar{g}(\nu,\bar{N})=0$, we have $\bar{g}(\vec{x},\bar{N})=0$ on $\partial \Sigma$.

\end{proof}

\begin{theorem}\label{teopringaphad} Let $\Sigma^n \hookrightarrow B_{\bar{r}_0} \subset B_{\bar{R}}$ be a free boundary minimal hypersurface, which satisfies 

\begin{equation}\label{gap0001}
\left|A\right|^2 \leq \frac{n^2}{4\bar{r}_0^2}.
\end{equation}
Then, $\Sigma$ is a totally geodesic disc through the origin.
\end{theorem}

\begin{proof} Consider the function $v=\gb(\x,\nb)$ defined on $\Sigma$, where $\nb$ is the normal vector to $\Sigma$. Considerer also the operator $L:= \Delta + \left|A\right|^2 + q$ where,
$$
q = \bar{R}ic(\bar{N},\bar{N}) + 4ne^{-2h}[u''(\left|x\right|^2)\left|x\right|^2+u'(\left|x\right|^2)].
$$

The $\textbf{Lemma \ref{limitrici}}$ ensures that $q\leq0$ and by hypothesis $(\ref{gap0001})$, the $\textbf{Lemma \ref{autovposi}}$ ensures that the first eigenvalue of the operator $L:= \Delta + \left|A\right|^2 + q $ for the problem $L[v_1]=-\lambda_1v_1$, with $v_1=0$ on the boundary $\partial \Sigma$, is strictly positive. Since in this case we have $L[v]=0$, it follows that $v\equiv0$ in $\Sigma$.  Thus, $0\equiv v =\bar{g}(\vec{x},\bar{N})=e^{2h}\left\langle \vec{x},\bar{N}\right\rangle$, and therefore $\left\langle \vec{x},\bar{N}\right\rangle \equiv 0$. The principal curvatures $k_i$ and $\bar{k}_i$, $i=1,...,n$ of $\Sigma$ with respect to the canonical metric $\left\langle, \right\rangle$ of $\R^{n+1}$ and $\bar{g}$ of $B_{\bar{R}}$ respectively, are related by the equation (see \textbf{Lemma} 10.1.1, \cite{Lop}),

$$\bar{k}_i=\frac{1}{e^h}\left(k_i-2u'(\left|\vec{x}\right|^2)\left\langle \vec{x},N\right\rangle\right)\quad \quad \quad i=1,...,n$$
where $N$ denotes the normal vector to $\Sigma$ in the metric $\left\langle , \right\rangle$ given by $N=e^h\bar{N}$. Since $\Sigma $ is minimal in the metric $\bar{g}$ we should have, 

\begin{eqnarray*}
0 = \sum_{i=1}^{n} \bar{k}_i &=& \sum_{i=1}^{n} \frac{1}{e^h}\left(k_i-2u'(\left|\vec{x}\right|^2)\left\langle \vec{x},N\right\rangle\right)\\
                             &=& \sum_{i=1}^{n} \frac{1}{e^h}\left(k_i-2u'(\left|\vec{x}\right|^2)e^h\left\langle \vec{x},\bar{N} \right\rangle\right) = \sum_{i=1}^{n} \frac{1}{e^h}k_i
\end{eqnarray*}
is that, $\sum_{i=1}^{n} k_i = 0$ and therefore $\Sigma$ is also a minimal hypersurface in $\R^{n+1}$ satisfying $\left\langle\vec{x},N \right\rangle\equiv 0 $. So, $\Sigma$ is a totally geodesic disc passing through the origin.

\end{proof}

Based on \textbf{remark \ref{obsteo}.}, we also have the following theorem, which proof is analogous to the proof of the above result.

\begin{theorem}\label{teopringaphad2} Let $\Sigma^n \hookrightarrow B_{\bar{r}_0} \subset B_{\bar{R}}$ be a free boundary minimal hypersurface, which satisfies 

\begin{equation}\label{gap0002}
\left|A\right|^2 \leq \frac{(n-2)^2}{4\bar{r}^2} 
\end{equation}
at every point  $x\in\Sigma\setminus \vec{0}$ where $\bar{r}=\bar{r}(x)$  is the distance in  $B_{\bar{R}}$ from a point $x \in \Sigma$ to the point $\vec{0}$. Then, $\Sigma$ is a totally geodesic disc through the origin.
\end{theorem}


\section{GAP Results for $\left|A\right|^2$ of free boundary minimal hypersurfaces in annular domain}

For $r_{_0} \leq \infty$ consider $B_{\bar{r}_{_0}}=(\B_{r_{_0}}^{n+1},\bar{g})$, where again, $\bar{g}=e^{2h}\left\langle ,\right\rangle$, with $h(x)=u(\left|x\right|^2)$ and $u$ under the conditions of \textbf{Lema \ref{condhad}}. For $r_1<r_2<r_{_0}$, define the $(n+1)-$dimensional annulus $\mathcal{A}(\bar{r}_1,\bar{r}_2):= B_{\bar{r}_2} \setminus B_{\bar{r}_1}$. When we want to emphasize that the metric $\bar{g}$ coincides with the canonical metric (that is, for $u\equiv0$) we write $\mathcal{A}(r_1,r_2)$ instead of $\anelconf$, unless otherwise stated. 

\begin{theorem}\label{teoprin1cap3} Let $\Sigma^n \hookrightarrow \mathcal{A}(\bar{r}_1,\bar{r}_2) $ be a immersed free boundary minimal hypersurface. Suppose that,

\begin{equation}\label{gapanel}
\left|A\right|^2 \leq \frac{n^2}{4\bar{r}_2^{2}}
\end{equation}
then, $\Sigma$ is tangent to the position vector field and furthermore the boundary $\partial \Sigma$ intersects the two connected components of the boundary $\partial \mathcal{A}(\bar{r}_1,\bar{r}_2)$.
\end{theorem} 

\begin{proof} The proof follows in analogous way to the \textbf{Theorem \ref{teopringaphad}}. We then conclude that the function $v=\bar{g}(\vec{x},\bar{N})$ must be identically null on $\Sigma$, which forces $\Sigma$ to be also a minimal hypersurface of free boundary in the domain $(\B_{r_2}^{n+1} \setminus \B_{r_1}^{n+1},\left\langle , \right\rangle)$. The difference here is that the fact of $\Sigma$ does not pass through the origin does not allow us to conclude that the hypersurface is totally geodesic. However, as we saw in the \textbf{Theorem \ref{teopringaphad}.}, $\bar{g}(\vec{x},\bar{N}) \equiv 0 $ implies that $\left\langle \vec{x},N \right\rangle \equiv 0$ when we consider $\Sigma$ as a hypersurface of $\R^{n+1}$ and thus, the boundary $\partial \Sigma$ necessarily intersects each of the connected components of the boundary $\partial \mathcal{A}(\bar{r}_1,\bar{r}_2)$, otherwise $\partial \Sigma$ would be in only one connected component of $\partial \mathcal{A}(\bar{r}_1,\bar{r}_2)$ and the function $f(x)=\left|x\right|^2$ restricted to $\Sigma$ would have a minimum point or a local maximum point in some point $x_0 \in \Sigma$, which implies that $\nabla f (x_0)=\vec{x}_0$ must be in the direction of the normal vector $N$ in $x_0$ and then $\left\langle \vec{x_0},N \right\rangle \neq 0$, which does not occur.
\end{proof}

The above theorem says that a minimal hypersurface $\Sigma^n \hookrightarrow \mathcal{A}(\bar{r}_1,\bar{r}_2)$ with free boundary satisfying the condition $(\ref{gapanel})$, necessarily has its support function  $v=\bar{g}(\vec{x},\bar{N})$ identically zero, and furthermore, the boundary $\partial \Sigma$ should intersect each of the connected components of $\partial \mathcal{A}$. Note that in the specific case of Euclidean space it is not possible to say immediately that such a hypersurface is totally geodesic, since for dimension $n\geq 3 $ the Euclidean space $\R^{n+1}$ admits minimal hypersurfaces not totally geodesic as some specific types of cones, 
whose portion that intersects the annulus $ \mathcal {A} (r_1, r_2) $ forms a minimal hypersurface with free boundary in $ \mathcal {A} (r_1, r_2)$. Note that such hypersurfaces have a singularity at the origin $\vec{0}$, so we must have $\left|A\right|^2(p)$ arbitrarily large when $p$ approaches to $\vec{0}$. That is why it is natural to expect that some restriction on the minor radius $r_1$ is necessary to characterize hypersurfaces satisfying $(\ref{gapanel})$ as totally geodesic. This is what we will do in the next sections.

\subsection{The case of a Euclidean annulus}

Let $C_{\Gamma}:=\{ \lambda y \ ; y \in \Gamma^{n-1}, \lambda\in (0,\infty)\}$ be a cone in $\R^{n+1}$ with vertex at origin. We refer to $C_{\Gamma}$ as the cone over ${\Gamma}$ for mean that $C_{\Gamma}$ intersects the sphere $\S^n$ along of the surface $\Gamma$. For our purposes we consider that $\Gamma^{n-1}$ is a closed orientable hypersurface on $\S^{n}$. Note that the support function $v=\left\langle \vec{x},N\right\rangle$ is such that $v\equiv 0$ on $C_{\Gamma}$, and if  $\Sigma^n \subset \R^{n+1}$ is a hypersurface which has the support function identically null then $\Sigma$ is contained in some cone $C_{\Lambda}$ for some hypersurface $\Lambda^{n-1} \subset \S^n$.

\begin{lemma}\label{conemin} $C_{\Gamma}$ is a minimal hypersurface in $\R^{n+1}$ if and only if $\Gamma$ is a minimal hypersurface in
 $\S^n$.  
\end{lemma}

\begin{proof} Let $\bar{\nabla}$ the connection of the $\R^{n+1}$ and N the normal field to $C_{\Gamma}$. The connection of $C_{\Gamma}$ is given by $\nabla_XY=(\bar{\nabla}_XY)^T$ for any X and Y fields tangents to $C_{\Gamma}$. The connection of $\Gamma$ as submanifold of $\S^n$ is given by $\nabla^{\Gamma}_ZW=(\nabla_ZW)^T$ for any Z and W tangents fields to $\Gamma$, seeing $\Gamma$ as a submanifold of $\S^n$. Let $p\in C_{\Gamma} \cap \Gamma^{n-1}$, then $\left|\vec{p}\right|=1$ and since $\vec{x} \in T_x C_{\Gamma}$ $\forall x \in C_{\Gamma}$, we have that $\vec{p} \in T_p C_{\Gamma}$. Now consider $E_n=\vec{p}$ and complete to a orthonormal basis $\{E_1,...,E_{n-1},E_n\}$ of $T_pC_{\Gamma}$. If $C_{\Gamma}$ is minimal in $\R^{n+1}$ then we should have,

\begin{equation}\label{curvmedcone}
\sumk \left\langle - \bar{\nabla}_{E_k}N, E_k \right\rangle(p) = 0.
\end{equation}

On the other hand,

$$
\left\langle \vec{x},N\right\rangle \equiv 0 \Rightarrow \vec{x}\left\langle\vec{x},N \right\rangle = 0 \Rightarrow \left\langle \bar{\nabla}_{_{\vec{x}}} \vec{x},N\right\rangle + \left\langle \vec{x},\bar{\nabla}_{_{\vec{x}}} N\right\rangle = 0  \Rightarrow \left\langle - \bar{\nabla}_{\vec{x}}N,\vec{x}\right\rangle=0,
$$

since $\left\langle \bar{\nabla}_{_{\vec{x}}} \vec{x},N\right\rangle = \left\langle\vec{x},N\right\rangle = 0$. In particular, for $\vec{x}=p$ we should have,

\begin{equation}\label{posidireprin}
\left\langle - \bar{\nabla}_{\vec{p}}N,\vec{p}\right\rangle=0 \Rightarrow \left\langle - \bar{\nabla}_{E_n} N, E_n \right\rangle =0.
\end{equation}

In this way, the equation $(\ref{curvmedcone})$ becomes, 

\begin{eqnarray}\label{curvmedcone2}
0 &=& \sumk \left\langle - \bar{\nabla}_{E_k}N, E_k \right\rangle(p) = \sum_{k=1}^{n-1} \left\langle - \bar{\nabla}_{E_k}N, E_k \right\rangle(p) + \left\langle - \bar{\nabla}_{E_n}N, E_n \right\rangle(p) \nonumber \\
  &=& \sum_{k=1}^{n-1} \left\langle - \bar{\nabla}_{E_k}N, E_k \right\rangle(p).
\end{eqnarray}

Since $p\in \S^n$, $\vec{p}$ is orthogonal to $T_p\S^n$ and being $\left\langle \vec{p},N\right\rangle=0$ it follows that $N\in T_p \S^n$. Then, $N$ is normal to $\Gamma^{n-1}$ as submanifold of $\S^n$. Being $\{E_1,...,E_{n-1}\}$ a basis for $T_p \Gamma$, the equation $(\ref{curvmedcone2})$ provides,

$$
\sum_{k=1}^{n-1} \left\langle - \nabla_{E_k} N, E_k \right\rangle(p) = \sum_{k=1}^{n-1} \left\langle - (\bar{\nabla}_{E_k} N)^T, E_k \right\rangle(p) =\sum_{k=1}^{n-1} \left\langle - \bar{\nabla}_{E_k} N, E_k \right\rangle(p)=0
$$

which says that $\Gamma^{n-1}$ is a minimal submanifold of  $\S^n$. The reciprocal is done in an analogous way. 

\end{proof}

\begin{example} For $n \geq 3$, consider the Clifford torus $\mathbb{T}_{m,n}:=\S_{\lambda_1}^{^m} \times \S_{\lambda_2}^{^{(n-1)-m}}$ where $\lambda_1=\sqrt{\frac{m}{n-1}}$, $\lambda_2 = \sqrt{\frac{(n-1)-m}{n-1}}$ and $1 \leq m \leq n-2$. Since $\mathbb{T}_{m,n}$ is a minimal  hypersurface in \ $\S^n$, the cone \ $C_{\mathbb{T}_{m,n}}=\{ \lambda y \ ; y \in \mathbb{T}_{m,n}, \lambda\in (0,\infty)\}$ is a minimal hypersurface in the Euclidean space $\R^{n+1}$.
\end{example}

Given $\lambda > 0 $ consider the hypersurface $\Gamma_{\lambda}:=C_{\Gamma} \cap \S_{\lambda}^n$ where $\S_{\lambda}^n$ denotes the sphere of radius $\lambda$. Note that $\Gamma_{\lambda}$ is obtained from  $\Gamma^{n-1}\subset \S^n$ by a homothety, that is, every point $x_0 \in \Gamma_{\lambda}$ is such that there exists a point  $p \in \Gamma^{n-1}$ in such a way that $x_0 = \lambda p$. Let $\left|A\right|^2$ and $\left|A_{_{\lambda}}\right|^2$ denoting the square of the second fundamental form of $C_{\Gamma}$ as the hypersurface of $\R^{n+1}$ and of $\Gamma_\lambda$ as the hypersurface of $\S^n_{\lambda}$, respectively.

\begin{lemma}\label{segform01} Let $\lambda > 0$ and consider $\Gamma_{\lambda} \hookrightarrow \S^{n}_{\lambda}$. We have, 

$$ 
\left|A\right|^2(q) = \left|A_{_{\lambda}}\right|^2(q)   
$$

\noindent for all  $q \in C_{\Gamma} \cap \Gamma_{\lambda}$. 
\end{lemma}

\begin{proof} Let $q \in \Gamma_{\lambda}$. Consider $E_n=\frac{\vec{q}}{\left|\vec{q}\right|}$ and complete to an orthonormal basis $\{E_1,...,E_{n-1},E_n\}$ of $T_q C_{\Gamma}$. As we have saw in $(\ref{posidireprin})$, the fact of $\left\langle \vec{x},N\right\rangle\equiv 0$ on $C_{\Gamma}$ ensures that $\left\langle - \bar{\nabla}_{\vec{x}} N, \vec{x} \right\rangle(q)=0$, therefore for $x=q$, we have $\left\langle - \bar{\nabla}_{\vec{q}} N, \vec{q} \right\rangle(q) = 0$ and multiplying this last expression by $\frac{1}{\left|\vec{q}\right|^2}$, we obtain $\left\langle - \bar{\nabla}_{E_n} N, E_n \right\rangle(q) = 0$. That is, the vector $- \bar{\nabla}_{E_n} N $ is orthogonal to $E_n$ or is null. We affirm that $- \bar{\nabla}_{E_n} N = 0$. Indeed, it is enough to observe that if we consider the curve $\alpha:(-\varepsilon,\varepsilon) \rightarrow C_{\Gamma}$ given by $\alpha(t)=(1+\frac{t}{\left|\vec{q}\right|}) q$ we have $\alpha'(0)=E_n$, as long as the tangent space $T_{\alpha(t)} C_{\Gamma}$ is the same for all $t$, we have that the normal vector N does not varies over $\Sigma$. Thereby, $0=\left\langle A(E_n), E_i \right\rangle(q) = \left\langle E_n,A(E_i)\right\rangle(q)$ for all $i=1,...,n$. Then, we have that

\begin{eqnarray}\label{quasegcon}
\left|A\right|^2(q) &=& \sumk \suml \left\langle A_{_C}(E_l), E_k \right\rangle^2 = \sum_{k=1}^{n-1} \suml \left\langle A_{_C}(E_l), E_k \right\rangle^2 + \suml \left\langle A_{_C}(E_l), E_n \right\rangle^2\nonumber \\
												 &=& \sum_{k=1}^{n-1} \left\{ \sum_{l=1}^{n-1} \left\langle A_{_C}(E_l), E_k \right\rangle^2 + \left\langle A_{_C}(E_n), E_k \right\rangle^2 \right\} = \sum_{k=1}^{n-1} \sum_{l=1}^{n-1} \left\langle A_{_C}(E_l), E_k \right\rangle^2\nonumber\\
												 &=& \sum_{k=1}^{n-1} \sum_{l=1}^{n-1} \left\langle - \bar{\nabla}_{E_l} N , E_k \right\rangle^2.
\end{eqnarray}

On the other hand, $\{E_1,...,E_{n-1}\}$ is an orthonormal basis of $T_{q}\Gamma_{\lambda}$ as submanifold of $\S^n_{\lambda}$, where $N$ is normal to $\Gamma_{\lambda}$ as subvariety of $\S^n_{\lambda}$, therefore,

\begin{eqnarray}\label{quaseggama} 
\left|A_{_{\lambda}}\right|^2(q) &=& \sum_{k=1}^{n-1} \sum_{l=1}^{n-1} \left\langle - \nabla_{E_l} N , E_k \right\rangle^2 = \sum_{k=1}^{n-1} \sum_{l=1}^{n-1} \left\langle - (\bar{\nabla}_{E_l} N)^T , E_k \right\rangle^2 \nonumber \\ 
                                &=&  \sum_{k=1}^{n-1} \sum_{l=1}^{n-1} \left\langle - \bar{\nabla}_{E_l} N , E_k \right\rangle^2,
\end{eqnarray}

\noindent where $\nabla$ denotes the connection of $\S^n$ with respect to the metric $\R^{n+1}$. Comparing the equations $(\ref{quasegcon})$ and $(\ref{quaseggama})$, follows the desired result. 

\end{proof}

\begin{lemma}\label{segform02} Consider $\Gamma^{n-1} \hookrightarrow \S^n$ a closed surface and be $C_{\Gamma}$, the cone in $\R^{n+1}$ over $\Gamma^{n-1}$. Given a $q \in C_{\Gamma}$ define $\lambda=\left|\vec{q}\right|$ and $\Gamma_\lambda=C_{\Gamma} \cap \S_{\lambda}^n$. Thus, for $p=\frac{1}{\lambda}q \in \S^n$ we have, 

\begin{equation}\label{segmudconf}
\left|A_{\lambda}\right|^2(q) = \frac{1}{\lambda^2}\left|A_{_{1}}\right|^2(p)
\end{equation}

\end{lemma}

\begin{remark}\label{obsconeminsse} A consequence of the identity $(\ref{segmudconf})$ is that if $\Gamma^{n-1}$ is a minimal hypersurface in $\S^n$, then the cone $C_{\Gamma}$ over $\Gamma^{n-1}$ is a totally geodesic minimal hypersurface in $\R^{n+1}$ if and only if $\Gamma$ is a totally geodesic minimal hypersurface in $\S^n$.  
\end{remark}

By lemmas $(\ref{segform01})$ and $(\ref{segform02})$ if $C_{\Gamma}$ is a cone in $\R^{n+1}$ over some hypersurface of $\S^n$ and $q \in C_{\Gamma}$ then, for $p=\frac{1}{\left|\vec{q}\right|}q$ we have, 

\begin{equation}\label{segform03}
\left|A\right|^2(q) = \frac{1}{\left|\vec{q}\right|^2}\left|A_{_{1}}\right|^2(p)
\end{equation}

This says that the square of the second fundamental form of $C_{\Gamma}$ as hypersurface of $\R^{n+1}$ when calculated on $\Gamma_{\lambda}= C_{\Gamma} \cap \S_{\lambda}^n$ for some $\lambda$, can be compared to the square of the second fundamental form of a hypersurface $\Gamma \hookrightarrow \S^n$ obtained by a homothety $\Gamma=\frac{1}{\lambda}\Gamma_{\lambda}$. In view of the \textbf{Lema \ref{conemin}.}, this comparison becomes useful in the context where we are working if the cone $C_{\Gamma}$ is a minimal hypersurface $\R^{n+1}$ due to the following theorem,


\begin{theorem}[Chern-do Carmo-Kobayashi, \cite{ChCaKo}]\label{ChernCarmoKoba}. Let $\Gamma^{n-1}$ be a closed minimal hypersurface in the unit sphere $\S^{n}$. Assume that its second fundamental form $A_1$ satisfies, 
$$
\left|A_1\right|^2 \leq n-1
$$
then,\\

(1) $\left|A_1\right|^2\equiv 0$ and $\Gamma^{n-1}$ is an equator $\S^{n-1} \subset \S^n$

(2) or $\left|A_1\right|^2 \equiv n-1$ and $\Gamma^{n-1}$ is one of Clifford tori $\mathbb{T}_{_{n,m}}$ 
\end{theorem}

\begin{example} For $n \geq 3$, let $C_{\mathbb{T}}$ be a minimal cone on $\R^{n+1}$ over the Clifford tori $\mathbb{T}_{m,n} \hookrightarrow \S^n$. For $r_2=1$ and $r_{_1}^2=\frac{4(n-1)}{n^2}r_{_2}^2$, we consider $\aneleu$ and let $C_{\mathbb{T}}(r_1,r_2)=C_{\mathbb{T}} \cap \mathcal{A}(r_1,r_2)$ be a portion of $C_{\mathbb{T}}$ inside on $\aneleu$. In view of the equation $(\ref{segform03})$, for any point $q \in C_{\mathbb{T}}(r_1,r_2)$ we have,
$$
\left|A\right|^2(q) = \frac{1}{\left|\vec{q}\right|^2}\left|A_{_{1}}\right|^2(p)
$$
But, by the above theorem, we have $\left|A_1\right|^2\equiv n-1$. Than, 
$$
\left|A\right|^2(q) = \frac{n-1}{\left|\vec{q}\right|^2} \leq \frac{n-1}{r_{_1}^2} = \frac{n^2}{4r_{_2}^2} 
$$
is that, 
$$
\left|A\right|^2(q) \leq \frac{n^2}{4r_{_2}^2} \quad \forall q \in C_{\mathbb{T}}(r_1,r_2)
$$
\end{example}

The preceding example says that condition $\left|A\right|^2(q) \leq \frac{n^2}{4r_{_2}^2}$ without any other assumption is not sufficient to characterize a free boundary minimal hypersurface on $\aneleu$ as being totally geodesic like was done on \textbf{Theorem \ref{teopringaphad}}. For this case we need a additional hypothesis which concerns about condition of distancing betwen the rays $r_1$ and $r_2$ of $\aneleu$ as we'll see on \textbf{Corollary \ref{classificacaoaneleu}}. But before that, we'll see a slightly more general case which will be useful to study a free boundary minimal hypersurfaces on $\anelconf$. 

\begin{proposition}\label{classifictroncocone} Let $C_{\Gamma}$ be a minimal cone in $\R^{n+1}$ with vertex on the origin and considered over a closed minimal hypersurface in $\S^n$. Consider $C_{\Gamma}(r_1,r_2)=C_{\Gamma} \cap \mathcal{A}(r_1,r_2)$ the trunk of cone $C_{\Gamma}$ inside of anullus $\mathcal{A}(r_1,r_2)$. Supose that for some constant $a_0$
\begin{equation}\label{gappanelprop}
\left|A\right|^2(q) \leq \frac{n^2}{4 r_{_2}^2} a_0 
\end{equation}
for all $q \in C_{\Gamma}(r_{_1},r_{_2})$. Hence,  \\

$i)$ If,
\begin{equation}\label{condraio}
r_{_1}^2 < \dfrac{4(n-1)}{n^2 a_0}\,r_{_2}^2 
\end{equation} 
than, $C_{\Gamma}(r_{_1},r_{_2})$ is the totaly geodesic anullus. In particular,  if the dimension $n$ satisfies $\frac{n^2}{n-1} \leq \frac{4}{a_{_0}}$ the inequality $(\ref{condraio})$ is always satisfied since $r_1<r_2$. Thus, just condition $(\ref{gappanelprop})$ must be satisfied so that $C_{\Gamma}$ be a totally geodesic disc.\\

$ii)$ if $\frac{n^2}{n-1} > \frac{4}{a_{_0}}$, 
\begin{equation}\label{condraio002}
r_{_1}^2 = \dfrac{4(n-1)}{n^2 a_0}\,r_{_2}^2 
\end{equation} 
and the equality $(\ref{gappanelprop})$ occur at some point  $q \in C_{\Gamma}(r_1,r_2)$, we have that $\Gamma$ is the Clifford Torus  $\mathbb{T}_{m,n}$.
\end{proposition}
\begin{proof} We consider $\Gamma_{\lambda}=C_{\Gamma}(r_1,r_2) \cap \S_{\lambda}^n$, where $r_1 \leq \lambda \leq r_2$, and we will denote by $A_{\lambda}$ the second fundamental form of respective hypersurfaces like a submanifolds of $\S_{\lambda}^n$. We note that $\Gamma=\frac{1}{\lambda}\Gamma_{\lambda}$. The second fundamental form of $\Gamma$ will be denoted by $A_1$. For each $q \in C_{\Gamma}(r_{_1},r_{_2})$ we choose $p=\frac{1}{\lambda}q \in \S^n$. As a consequence of $(\ref{segform03})$, 
$$
\left|A\right|^2(q) = \frac{1}{\left|q\right|^2}\left|A_{_{1}}\right|^2(p) 
$$
is that,
\begin{equation}\label{segform04}
\left|A_{_{1}}\right|^2(p) = \left|A\right|^2(q) \left|q\right|^2   
\end{equation}
for all $q \in C_{\Gamma}(r_1,r_2)$.\\
$i) (\Rightarrow) $ Thus, by  hypothesis $(\ref{gappanelprop})$ and the condition $(\ref{condraio})$ on the rays $r_1$ and $r_2$ we have, 
$$
\left|A_{_{1}}\right|^2(p) = \left|A\right|^2(q) \left|q\right|^2 \leq \frac{n^2}{4 r_{_2}^2} a_0 \left|q\right|^2 < \frac{n-1}{r_{_1}^2}\left|q\right|^2
$$
is that, 
$$
\left|A_{_{1}}\right|^2(p) < \frac{n-1}{r_{_1}^2}\left|q\right|^2
$$
The inequality above is true for all $q \in C_{\Gamma}(r_1,r_2)$ and $p \in \Gamma$ such that $p=\frac{1}{\lambda}q$, being $r_1 \leq \left|q\right| \leq r_2$, in particular for $\left|q\right|=r_1$ we have,
$$
\left|A_{_{1}}\right|^2(p) < n-1
$$ 
for all $p \in \Gamma$. By \textbf{Theorem \ref{ChernCarmoKoba}.} follows that $\Gamma \hookrightarrow \S^n$ is a totaly geodesic hypersurface, and by \textbf{Remark \ref{obsconeminsse}.}, $C_{\Gamma}(r_1,r_2)$ is a totally geodesic minimal hypersurface.\\ 

$ii) (\Rightarrow )$ Developing the equation $(\ref{segform04})$ and apply the hypothesis $(\ref{gappanelprop})$ and $(\ref{condraio002})$ we have, 
\begin{equation}
\left|A_{_{1}}\right|^2(p) = \left|A\right|^2(q) \left|q\right|^2 \leq \frac{n^2}{4 r_{_2}^2} a_0 \left|q\right|^2 = \frac{n-1}{r_{_1}^2}\left|q\right|^2
\end{equation}
Again, the equation is satifies for all $q \in C_{\Gamma}(r_1,r_2)$ and $p \in \Gamma$ such that $p=\frac{1}{\lambda}q$, being $r_1 \leq \left|q\right| \leq r_2$, in particular for $\left|q\right|=r_1$ we have,
\begin{equation}\label{desisegform}
\left|A_{_{1}}\right|^2(p) \leq n-1
\end{equation}
but, by hypothesis, the equality on $(\ref{gappanelprop})$ occur for some $q \in C_{\Gamma}(r_1,r_2)$, therefore must occur too on $(\ref{desisegform})$ for some $p \in \Gamma$. Again, the \textbf{Theorem \ref{ChernCarmoKoba}.} ensures that $\Gamma \hookrightarrow \S^n$ is a Clifford torus and the support cone of the trunk $C_{\Gamma}(r_1,r_2)$ is considered over on such torus. 
\end{proof}

\begin{corollary}(Teorema $\ref{teoprin1cap3}$)\label{classificacaoaneleu} Let $\Sigma^n \hookrightarrow \mathcal{A}(r_1,r_2)$ be an immersed free boundary minimal hypersurface in Euclidean anullus $(n+1)$-dimensional. Assume that, 

\begin{equation}\label{gapanel2}
\left|A\right|^2 \leq \frac{n^2}{4r_{_2}^{2}}
\end{equation}
than, 

$i)$ If $n = 2$ then $\Sigma^n \hookrightarrow \mathcal{A}(r_1,r_2)$ is the totally geodesic annulus.

$ii)$ If $n\geq3$ and $r_{_1}^2 < \frac{4(n-1)}{n^2}\,r_{_2}^2$ then $\Sigma^n \hookrightarrow \mathcal{A}(r_1,r_2)$ is a totally geodesic annulus. 

$iii)$ If $n\geq3$ and $r_{_1}^2 = \frac{4(n-1)}{n^2}\,r_{_2}^2$ then either $\Sigma^n \hookrightarrow \mathcal{A}(r_1,r_2)$ is a totally geodesic annulus or $\Sigma^n$ is a trunk of cone whose support cone is considered over the Clifford Torus in  $\S^n$.
\end{corollary}

\begin{proof} By \textbf{Teorema \ref{teoprin1cap3}.} the support function $v=\left\langle \vec{x},N\right\rangle$ is identically null on $\Sigma^n$ and the boundary $\partial \Sigma$ intersect the two connected components of $\partial \mathcal{A}(r_1,r_2)$. In summary, we have that $\Sigma$ is a subset of a minimal cone in $\R^{n+1}$. The conclusions follows of  \textbf{Proposition \ref{classifictroncocone}.} by choose of $a_{_0}=1$.  
\end{proof}

\subsection{The case of annulus conformal to euclidian annulus}

Now we will study the case of a immersion of free boundary minimal hypersurface $\Sigma^n \hookrightarrow \anelconf$ where the annulus $\mathcal{A}$ is not necessarily euclidean. When $\Sigma^n$ is equipped with euclidean geometry of $\B^{n+1}_{r_{_2}} \setminus \B^{n+1}_{r_{_1}}$ we'll write $\Sigma^{n}_{\delta}$, and the second fundamental form of the respective geometry will denoted by $A_{\delta}$. If $p \in \Sigma^n$ has euclidean distance up to origin given by $r=\left|p\right|$, then the distance of $p$ up to origin with respect the metric $\bar{g}=e^{2u(\left|x\right|^2)}\left\langle,  \right\rangle$ is given by equation,
\begin{equation}\label{raioconfint}
\bar{r}=rI(r) 
\end{equation}
where $I(r)=\displaystyle \int_{0}^{1} e^{u(t^2r^2)} dt$. For the following purposes, we define,
$$
m_0=\sup\{e^{2u(\left|x\right|^2)}; x \in \aneleu \}
$$
\begin{corollary}(Teorema $\ref{teoprin1cap3}$)\label{classificconf} Let $\Sigma^n \hookrightarrow \anelconf$ be a immersed free boundary minimal hypersurface in $(n+1)$-dimensional annulus conformal to euclidean annulus $\aneleu$. Assume that, 
\begin{equation}\label{gapanel3}
\left|A\right|^2 \leq \frac{n^2}{4\bar{r}_{_2}^{2}}
\end{equation}
than, 

$\textbf{1})$ If $\frac{n^2}{(n-1)} \leq 4\left(\dfrac{I(r_{_2})^2}{m_0} \right)$, $\Sigma^n \hookrightarrow \anelconf$ is a totally geodesic disc.\\ 

$\textbf{2})$ If $ \frac{n^2}{(n-1)} > 4\left(\dfrac{I(r_{_2})^2}{m_0} \right)$ and $r_{_1}^2 < \dfrac{4(n-1)}{n^2}\left( \dfrac{I(r_{_2})^2}{m_0}\right)\,r_{_2}^2$, $\Sigma^n \hookrightarrow \anelconf$ is the totally geodesic disc.\\

$\textbf{3})$ If $\frac{n^2}{(n-1)} > 4\left(\dfrac{I(r_{_2})^2}{m_0} \right)$ and $r_{_1}^2 = \dfrac{4(n-1)}{n^2}\left( \dfrac{I(r_{_2})^2}{m_0}\right)\,r_{_2}^2$ then either $\Sigma^n \hookrightarrow \anelconf$ is totally geodesic annulus or $\Sigma^n$ is a trunk of cone whose support cone is considered over a Clifford torus in $\S^n$ when considered like $\Sigma^n \hookrightarrow \aneleu$
\end{corollary}
\begin{proof} 
By \textbf{Theorem \ref{teoprin1cap3}.} the support function $v=\bar{g}(\vec{x},\bar{N})=e^{2h}\left\langle \vec{x},\bar{N}\right\rangle$ is identically null on $\Sigma^n$ and the boundary $\partial \Sigma$ intersect the two connected components of $\partial \anelconf$. The principal curvatures $\bar{k}_i$ e $k_i$ of $\Sigma^n$ and $\Sigma_{\delta}^n$ respectively, are related by the equation,
\begin{equation}\label{cruvprincanel}
\bar{k}_i=\frac{1}{e^h}k_i - 2u'(\left|x\right|^2)\left\langle \vec{x},N\right\rangle
\end{equation}
where $N$ is the normal vector to $\Sigma_{\delta}^n$ given by $N=e^h\bar{N}$.  As $v \equiv 0$ we have $0\equiv \left\langle \vec{x},\bar{N}\right\rangle=e^{-h}\left\langle \vec{x},N\right\rangle$, therefore $\left\langle \vec{x},N\right\rangle \equiv 0$. Thus, the equation $(\ref{cruvprincanel})$ says that $\Sigma^n$ is a free boundary minimal hypersurface in $\aneleu$ therefore $\Sigma^n$ is a minimal trunk of cone $C_n$ considered over some minimal hypersurface $\Gamma_1 \hookrightarrow \S^n$ and moreover, 
$$
\bar{k}_i=\frac{1}{e^{h}}k_i \Rightarrow e^{2h}\left|A\right|^2=\left|A_{\delta}\right|^2
$$
where $A_{\delta}$ denotes the second fundamental form of $\Sigma_{\delta}^n$. By definition of $m_0$ we have,
$$
\left|A_{\delta}\right|^2=\left|A\right|^2e^{2h} \leq \left|A\right|^2 m_0 
$$
Now by the hypothesis on $\left|A\right|^2$ we have,
$$
\left|A_{\delta}\right|^2 \leq \left|A\right|^2 m_0 \leq \frac{n^2}{4\bar{r}_{_2}^{2}} m_0 =\frac{n^2}{4r_{_2}^{2}} \frac{r_{_2}^2}{\bar{r}_{_2}^{2}}m_0  = \frac{n^2}{4r_{_2}^{2}} a_0
$$
where $a_0=\dfrac{r_{_2}^2}{\bar{r}_{_2}^{2}}m_0$. Thus,
\begin{equation}\label{gapaux001}
\left|A_{\delta}\right|^2 \leq \frac{n^2}{4r_{_2}^{2}} a_0
\end{equation}
As $I(r_2)^2=\dfrac{\bar{r}_{_2}^{2}}{r_{_2}^2}$, we have $a_0=\dfrac{m_0}{I(r_2)^2}$ therefore,
\begin{equation}\label{gapaux002}
\frac{n^2}{(n-1)} \leq 4\left(\dfrac{I(r_{_2})^2}{m_0} \right)\,\,  \Leftrightarrow \,\, \frac{n^2}{(n-1)} \leq \dfrac{4}{a_0}\,
\end{equation}
Thus, $\Sigma^n$ is a trunk of minimal cone with free boundary in $\aneleu$ which satisfies the conditions $(\ref{gapaux001})$ and $(\ref{gapaux002})$. By item $i)$ of \textbf{Proposition \ref{classifictroncocone}.}, $\Sigma^n \hookrightarrow \aneleu$ is a totally geodesic annulus , i.e $\left|A_{\delta}\right|^2\equiv 0$, but $\left|A_{\delta}\right|^2=\left|A\right|^2e^{2h}$ therefore, $\left|A\right|^2=0$ and follows that $\Sigma^n \hookrightarrow \anelconf$ is a totally geodesic and the item \textbf{1} follows. Moreover we have,
\begin{equation}\label{gapaux003}
\frac{n^2}{(n-1)} > 4\left(\dfrac{I(r_{_2})^2}{m_0} \right)  \,\,  \Leftrightarrow \,\, \frac{n^2}{(n-1)} > \dfrac{4}{a_0}\, 
\end{equation}
further, 
\begin{equation}\label{gapaux004}
r_{_1}^2 < (=) \dfrac{4(n-1)}{n^2}\left( \dfrac{I(r_{_2})^2}{m_0}\right)\,r_{_2}^2 \Leftrightarrow r_{_1}^2 < (=) \dfrac{4(n-1)}{n^2 a_0}r_{_2}^2 
\end{equation}
By the analogous way, the item $i)$ of \textbf{Proposition \ref{classifictroncocone}.} ensure by contitions $(\ref{gapaux001})$, $(\ref{gapaux003})$ and the strict inequality on both $(\ref{gapaux004})$, which $\Sigma^n \hookrightarrow \aneleu$ is a totally geodesic annulus, and by the same previous argument we have that $\Sigma^n \hookrightarrow \anelconf$ is totally geodesic annulus and the item \textbf{2} follows. Finally, the item $ii)$ of the same proposition ensure by conditions $(\ref{gapaux001})$, $(\ref{gapaux003})$ with equality on $(\ref{gapaux004})$ wich one of the situation occur; either $\left|A_{\delta}\right|^2\equiv\left|A\right|^2\equiv0$ or $\Sigma^n \hookrightarrow \aneleu$ is a trunk of cone whose support cone is considered over a Clifford torus and the item \textbf{3} follows. 
\end{proof}
\begin{example} 
Let $\H^{n+1}$ be hyperbolic space modeled by Poincaré disc; $\H^{n+1}=(\B^{n+1}_1,\bar{g})$, where $\bar{g}=e^{2h}\left\langle , \right\rangle$ with $h(x)=u(\left|x\right|^2)=\ln\left(\frac{2}{1-\left|x\right|^2}\right)$. Consider $r_1<r_2<1$ and define $\anelconf:=(\aneleu,\bar{g})$. By $(\ref{raioconfint})$ we have, 
$$
\bar{r}_i=r_i\frac{2\tanh^{-1}(r_i)}{r_i}=2\tanh^{-1}(r_i)
$$
For this case, $m_0=\dfrac{4}{(1-r_{_2}^2)^2}$. Define, 
$$
f(r):=\dfrac{4I(r)^2}{m_0} = 4[\tanh^{-1}(r)]^2\left(\frac{1-r^2}{r}\right)^2
$$
the function $f:(0,1)\rightarrow\R$ satisfies $\lim_{r\rightarrow 0}f(r)=4$, $\lim_{r\rightarrow 1}f(r)=0$ and $f'(r)<0$. Thus, $f$ is  decresing, and $0 \leq f(r) \leq 4$.

Let $\Sigma^n \hookrightarrow \anelconf$ be a free boundary minimal surface such that,
$$
\left|A\right|^2 \leq \frac{n^2}{4\bar{r}_{_2}^{2}}
$$
Note that for  $n \geq 2$ we have, 
$$ 
\frac{n^2}{(n-1)} > f(r)=\dfrac{4I(r)^2}{m_0}
$$
Thus,

$A)$ If $r_{_1}^2 < \dfrac{(n-1)}{n^2}f(r_{_2})r_{_2}^2$, than $\Sigma^n$ is a totally geodesic annulus.\\

$B)$ If $r_{_1}^2 = \dfrac{(n-1)}{n^2}f(r_{_2})r_{_2}^2$ than, $\Sigma^n$ is a totally geodesic annulus for $n=2$ and if $n\geq3$ we have two situations; either $\Sigma^n$ is a totally geodesic annulus or a trunk of cone whose support cone is considered over the Clifford torus $\mathbb{T}_{m,n}\hookrightarrow \S^n$ when we consider $\Sigma^n \hookrightarrow \aneleu$. 
\end{example}

\end{document}